\documentclass[10pt,a4paper] {article} 
\usepackage{latexsym,amsmath,enumerate,amssymb,amsbsy,amsthm} 
\usepackage{enumerate,verbatim,tikz} 

\newtheorem{theorem}{Theorem}[section] 
\newtheorem{lemma}[theorem]{Lemma} 
 
\newtheorem{proposition}[theorem]{Proposition} 
\newtheorem{corollary}[theorem]{Corollary} \numberwithin{equation}{section} 

\theoremstyle{definition}
 
\newtheorem{remark}[theorem]{Remark} 
\newtheorem{example}[theorem]{Example}

\title{Hermitian Self-Dual Cyclic Codes of Length $p^a$ over ${\rm GR}({p^2},s)$} 
\author{Somphong~Jitman, San~Ling, and Ekkasit~Sangwisut\thanks{S. Jitman is with the  Department of Mathematics, Faculty of Science, 
              Silpakorn University, Nakhon Pathom 73000,  Thailand
  (email: sjitman@gmail.com).} 
\thanks{S. Ling is with the Division of Mathematical Sciences,
  School of Physical and Mathematical Sciences, Nanyang Technological
  University,  21 Nanyang Link, Singapore 637371,
  Republic of Singapore
  (email:  lingsan@ntu.edu.sg).} 
\thanks{E. Sangwisut is  with the Department of Mathematics and Computer Science, Faculty of Science, Chulalongkorn University, Bangkok 10330, Thailand  (email:  e.sangwisut@gmail.com).} 
\thanks{This work is partially   supported by the National Research Foundation of Singapore under Research Grant NRF-CRP2-2007-03. E. Sangwisut is also supported by the Institute for the Promotion of Teaching Science and Technology of Thailand.}} 
\date{}

\renewcommand\footnotemark{}
\begin{document} 
\maketitle 
\begin{abstract}

	In this paper, we study cyclic codes  over the Galois ring ${\rm GR}({p^2},s)$. The main result is the  characterization and  enumeration of Hermitian self-dual cyclic codes of length $p^a$ over ${\rm GR}({p^2},s)$. Combining with some known results  and the standard  Discrete Fourier Transform decomposition, we arrive at  the characterization and enumeration of Euclidean self-dual cyclic codes of any length over ${\rm GR}({p^2},s)$. Some corrections to results on Euclidean self-dual cyclic codes of even length over $\mathbb{Z}_4$ in Discrete Appl. Math. 128, (2003), 27  and Des. Codes Cryptogr. 39, (2006), 127 are provided. 
\end{abstract}

\section{Introduction}

Cyclic and self-dual codes    over finite fields have  been extensively  studied for both   theoretical and practical reasons  (see \cite{JLX2011},  \cite{NRS2006},  and references therein). These concepts have   been extended and studied   over the  ring $\mathbb{Z}_4$  (see \cite{AO2003}, \cite{B2003},  and \cite{DL2006}), after it has been proven that some binary non-linear codes, such as the Kerdock, Preparata, and Goethals codes, are the Gray image of linear cyclic codes over $\mathbb{Z}_4$ in \cite{HKCSS1994}. Later on,  the study of cyclic and self-dual codes  has    been generalized to codes over $\mathbb{Z}_{p^r}$ and  Galois rings (see \cite{DP2007}, \cite{KLL2008}, \cite{KLL2012}, \cite{SE2009}, and  \cite{SE2009-2}). 

In \cite{B2003}, the structure of cyclic and Euclidean self-dual cyclic codes of oddly even  length ($2m$, where $m$ is odd) over $\mathbb{Z}_4$ has been studied  via the Discrete Fourier Transform decomposition. This idea has been extended to   the case of all even lengths in \cite{DL2006}. However,  some   results concerning Euclidean self-dual cyclic codes over $\mathbb{Z}_4$ in \cite{B2003} and   \cite{DL2006}  are not correct. The corrections to these are provided in Section \ref{sec4}.   Using a spectral approach and a generalization of the results in  \cite{B2003} and \cite{DL2006},  the structure of   cyclic codes over $\mathbb{Z}_{p^e}$, for any prime $p$, has been studied in  \cite{DP2007}. 
A nice classification of cyclic and Euclidean self-dual cyclic codes of length $p^a$ over ${\rm GR}({p^2},s)$ has been given  using a different  approach in \cite{KLL2008}, \cite{KLL2012},  and  \cite{SE2009}.

In this paper, we focus on the  characterization and enumeration   of Hermitian self-dual cyclic codes of length $p^a$ over ${\rm GR}({p^2},s)$ and their application to  the enumeration of  Euclidean self-dual cyclic codes of any length over ${\rm GR}({p^2},s)$. Using the standard Discrete Fourier Transform decomposition viewed as an extension of \cite{DP2007}, a  cyclic code  $C$ of any length $n=mp^a$, where $p$ is a prime and $p\nmid m$,  over  ${\rm GR}({p^2},s)$ can be viewed as a product of cyclic codes of length $p^a$ over some Galois extensions of ${\rm GR}({p^2},s)$.  Euclidean self-dual cyclic codes  can be characterized based on  this decomposition.
Applying some  known results concerning cyclic and Euclidean self-dual cyclic codes of length $p^a$ in \cite{KLL2008} and \cite{KLL2012} and our result on Hermitian self-dual cyclic codes of length $p^a$, the number of Euclidean self-dual cyclic codes of arbitrary length over ${\rm GR}({p^2},s)$ can be determined. Finally, we point out some mistakes   on  Euclidean self-dual cyclic codes of even length over $\mathbb{Z}_4$ in \cite{B2003}  and \cite{DL2006}.  The corrections to these are given as well.

The paper is organized as follows. Some preliminary concepts and results are recalled in Section \ref{sec2}. 
In Section \ref{sec3}, we prove the main result concerning the number of Hermitian self-dual cyclic codes of length $p^a$ over ${\rm GR}({p^2},s)$.  An application to the  enumeration of  Euclidean self-dual cyclic codes of any length over ${\rm GR}({p^2},s)$ and corrections to  \cite{B2003} and \cite{DL2006} are discussed in Section \ref{sec4}. A conclusion is provided  in Section \ref{sec5}.

\section{Preliminaries}  \label{sec2}
In this section, we  recall some definitions and basic properties of cyclic codes over  the  Galois ring ${\rm GR}(p^2,s)$.  

\subsection{Cyclic Codes over ${\rm GR}(p^2,s)$}
  For a prime $p$ and a positive integer $s$, the  Galois ring ${\rm GR}(p^2,s)$ is the Galois extension of the integer residue ring $\mathbb{Z}_{p^2}$ of degree $s$.   Let $\xi$ be an element in ${\rm GR}(p^2,s)$ that  generates the Teichm\"{u}ller set $\mathcal{T}_s$ of  ${\rm GR}(p^2,s)$. In other words, $\mathcal{T}_s=\{0,1,\xi,\xi^2,\dots, \xi^{p^{s}-2}\}$.
Then every  element  in  ${\rm GR}(p^2,s)$   has a unique $p$-adic expansion of the form
\[\alpha=a+bp,\]
where $a,b\in \mathcal{T}_s$.  If $s$ is even,  let $\bar{~}$  denote the automorphism 
on   ${\rm GR}(p^2,s)$ defined by
\begin{align}\label{aut}
	\overline{\alpha}=a^{p^{s/2}}+b^{p^{{s}/{2}}}p.
	\end{align} 
For more details  concerning  Galois rings, we refer the readers to \cite{W2003}.  

A  {\em cyclic code}  of length $n$ over ${\rm GR}(p^2,s)$ is a ${\rm GR}(p^2,s)$-submodule of the ${\rm GR}(p^2,s)$-module $({\rm GR}(p^2,s))^n$ which is invariant under the cyclic shift.      It is well known that every  cyclic code $C$ of length $n$ over ${\rm GR}(p^2,s)$ can   be regarded as an ideal in the quotient polynomial ring ${\rm GR}(p^2,s)[X]/\langle X^n-1\rangle$ and  represented by  its polynomial  representation \[\left\{\sum_{i=0}^{n-1}c_iX^i \,\middle\vert\, (c_0,c_1,\dots, c_{n-1})\in C\right\}.\]

For a given cyclic   code $C$ of length $n$ over  ${\rm GR}(p^2,s)$,  denote by $C^{\perp_{\rm E}}$ the   {\em Euclidean dual} of  $C$  defined with respect to the form 
\begin{align*}
	\langle \boldsymbol{u},\boldsymbol{v}\rangle_{\rm E}:=\sum_{i=0}^{n-1} u_iv_i
\end{align*}
where  $ \boldsymbol{u}= \sum_{i=0}^{n-1} u_iX^i $ and $ \boldsymbol{v}= \sum_{i=0}^{n-1} v_iX^i$.
The code $C$ is said to be {\em Euclidean self-dual} if  $C=C^{\perp_{\rm E}}$. 

In addition, if $s$ is even,  
we can also consider the  {\em Hermitian dual}  $C^{\perp_{\rm H}}$ of  $C$  defined with respect to the form 
\begin{align*}
	\langle \boldsymbol{u},\boldsymbol{v}\rangle_{\rm H}:=\sum_{i=0}^{n-1} u_i\overline{v_i}.
\end{align*} 
The code $C$ is said to be {\em Hermitian self-dual} if  $C=C^{\perp_{\rm H}}$.

The goal of this paper is to characterize and  enumerate   self-dual codes over ${\rm GR}(p^2,s)$.
For convenience, let 	$N({\rm GR}(p^2,s),n)$, $N_{\rm E}({\rm GR}(p^2,s),n)$, and $N_{\rm H}({\rm GR}(p^2,s),n)$ denote the numbers of cyclic codes, Euclidean self-dual cyclic codes, and Hermitian  self-dual cyclic codes of length $n$ over  ${\rm GR}(p^2,s)$, respectively.

\subsection{Some Results on Cyclic Codes of Length $p^a$ over ${\rm GR}(p^2,s)$}
In this subsection, we recall some results concerning cyclic and Euclidean self-dual cyclic codes of length $p^a$ over ${\rm GR}(p^2,s)$ in terms of  ideals in ${\rm GR}(p^2,s)[X]/\langle X^{p^a}-1\rangle$.

In \cite{KLL2008}, it has been shown that all the ideals in ${\rm GR}(p^2,s)[X]/\langle X^{p^a}-1\rangle$ have a unique representation. 
\begin{proposition}
	[{\cite[Theorem 3.8]{KLL2008}}]\label{Structure-cyclic} Every ideal in ${\rm GR}(p^2,s)[X]/\langle X^{p^a}-1\rangle$ can be uniquely represented in the form of 
	\begin{align}
		\label{cyclic} C=\langle (X-1)^{i_0}+p\sum_{j=0}^{i_1-1}h_j(X-1)^j, p(X-1)^{i_1} \rangle, 
	\end{align}
	where $i_0,i_1$ are integers such that $0\leq i_0<p^a$, $0\leq i_1\leq \min\{i_0, p^{a-1}\}$, $i_0+i_1\leq p^a$, and $h_j$ is an element in the Teichm\"{u}ller set $ \mathcal{T}_s$ for all $j$. 
\end{proposition}
The integer $i_1$ in Proposition~\ref{Structure-cyclic} is called the {\em first torsion index} of $C$.
\begin{corollary}
	[{\cite[Corollary 3.9]{KLL2008}}]\label{N-cyclic1} In ${\rm GR}(p^2,s)[X]/\langle X^{p^a}-1\rangle$, the number of distinct ideals with $i_0+i_1=d$, where $0\leq d\leq p^a$, is
	\[\frac{p^{s({\ell}+1)}-1}{p^s-1},\]
	where $\ell=\min\{\lfloor \frac{d}{2}\rfloor, p^{a-1}\}$. 
\end{corollary}

The next corollary follows immediately from Corollary~\ref{N-cyclic1} and \cite[Theorem 3.6]{KLL2008}. 
\begin{corollary}
	\label{N-cyclic2} The number of all cyclic codes of length $p^a$ over ${\rm GR}(p^2,s)$ is 
	\begin{align}
		N({\rm GR}(p^2,s),p^a)=2\left(\sum_{d=0}^{p^a-1} \frac{p^{s( \min\{\lfloor \frac{d}{2}\rfloor, p^{a-1}\} +1)}-1}{p^s-1}\right) +\frac{p^{s( p^{a-1} +1)}-1}{p^s-1}. 
	\end{align}
\end{corollary}

Combining the results in \cite{KLL2008}, \cite{KLL2012}, and \cite{SE2009},   the complete  enumeration of Euclidean self-dual cyclic codes of length $p^a$ over ${\rm GR}(p^2,s)$  can be  summarized as follows. 
\begin{proposition}
	[{\cite[Corollary 3.5]{KLL2012}}]\label{NEC} The number of Euclidean self-dual cyclic codes of length $2^a$ over ${\rm GR}(2^2,s)$ is
	\[N_{\rm E}({\rm GR}(2^2,s),2^a)= 
	\begin{cases}
		1 & \text{ if }a=1,\\
		1+2^s & \text{ if }a=2,\\
		1+2^s+2^{2s+1}\left( \frac{{(2^s)}^{(2^{a-2}-1)}-1}{2^s-1}\right)& \text{ if }a\geq 3. 
	\end{cases}
	\]
	
	If $p$ is an odd prime, then the number of Euclidean self-dual cyclic codes of length $p^a$ over ${\rm GR}(p^2,s)$ is
	\[N_{\rm E}({\rm GR}(p^2,s),p^a)= 2\left( \frac{{(p^s)}^\frac{p^{a-1}+1}{2}-1}{p^s-1}\right).\]
\end{proposition}

\section{Hermitian Self-Dual Cyclic Codes of Length $p^a$ over ${\rm GR}(p^2,s)$} \label{sec3}

In this section, we assume that $s$ is even and focus on characterizing and enumerating Hermitian self-dual cyclic codes of length $p^a$ over ${\rm GR}(p^2,s)$. 

It has been proven in \cite[Theorem 2]{SE2009} that the Euclidean dual of the cyclic code $C$ in (\ref{cyclic}) is of the form
\begin{align*}
	C^{\perp_{\rm E}}=&\left\langle (X-1)^{p^a-i_1}-p(X-1)^{p^a-i_0-i_1} \sum_{t=0}^{i_1-1}\left( \sum_{j=0}^t (-1)^{i_0+j} { i_0-j\choose t-j} h_j\right) (X-1)^t \right.\\
	&\left.+ \sum_{t=1}^{\mathcal{K}}\left( \sum_{j=1}^{\min\{t,p-1\}} (-1)^{j+1} { p-j\choose t-j} {p \choose j}\right) (X-1)^{tp^{a-1}-i_1}, p(X-1)^{p^a-i_0} \right\rangle, 
\end{align*}
where $\mathcal{K}=\lfloor\frac{p^a-i_0+i_1-1}{p^{a-1}} \rfloor$ and the empty sum is regarded as zero.

For a subset $A$ of   ${\rm GR}(p^2,s)[X]/\langle X^{p^a}-1\rangle$, let  $\overline{A}$ denote the set 
\[\left\{\sum_{i=0}^{p^a-1}\overline{a_i}X^i \,\middle\vert\, \sum_{i=0}^{p^a-1}{a_i}X^i \in A\right\},\]
 where $\bar{~}$ is the automorphism defined in (\ref{aut}). Since it is well know that $C^{\perp_{\rm H}}
	=\overline{	C^{\perp_{\rm E}}} $, we  have
\begin{align*}
	C^{\perp_{\rm H}}
	=&\left\langle (X-1)^{p^a-i_1}-p(X-1)^{p^a-i_0-i_1} \sum_{t=0}^{i_1-1}\left( \sum_{j=0}^t (-1)^{i_0+j} { i_0-j\choose r-j} {h_j}^{p^{{s}/{2}}}\right) (X-1)^t\right. \\
	&\left.+ \sum_{t=1}^{\mathcal{K}}\left( \sum_{j=1}^{\min\{t,p-1\}} (-1)^{j+1} { p-j\choose t-j} {p \choose j}\right) (X-1)^{tp^{a-1}-i_1}, p(X-1)^{p^a-i_0} \right\rangle.
\end{align*}

If $C=C^{\perp_{\rm H}}$, then $|C|=(p^s)^{p^a}$ which implies that $i_0+i_1=p^a$. Then we can write 
\begin{align}
	\label{genCH} C^{\perp_{\rm H}}= 
	\begin{cases}
		\left\langle (X-1)^{i_0}-p \sum_{t=0}^{i_1-1}\left( \sum_{j=0}^t (-1)^{i_0+j} { i_0-j\choose t-j} {h_j}^{p^{{s}/{2}}}\right) (X-1)^t,  p(X-1)^{i_1} \right\rangle~~~\\
		\hfill \text{ if }i_1<\lfloor \frac{p^{a-1}+1}{2}\rfloor,\\
		\left\langle (X-1)^{i_0}-p\sum_{t=0}^{i_1-1}\left( \sum_{j=0}^t (-1)^{i_0+j} { i_0-j\choose r-j}  {h_j}^{p^{{s}/{2}}}\right) (X-1)^t\right.\\
		\left. +p(X-1)^{p^{a-1}-i_1}, p(X-1)^{i_1} \right\rangle \hfill\text{ if }i_1\geq \lfloor \frac{p^{a-1}+1}{2}\rfloor. 
	\end{cases}
\end{align}

If $i_1=0$, then $ p {\rm GR}(p^2,s)[X]/\langle X^{p^a}-1\rangle$ is the only Hermitian self-dual cyclic code of length $p^a$ over ${\rm GR}(p^2,s)$. Next, we assume that $i_1\geq 1$.

By the unique representation of $C=C^{\perp_{\rm H}}$, (\ref{cyclic}), and (\ref{genCH}), we have 
\begin{align}
	\label{setupH} p {h_t} ^{p^{{s}/{2}}}= p\left(b_t +\sum_{j=0}^t (-1)^{i_0+j-1} { i_0-j\choose r-j} {h_j} \right) 
\end{align}
for all $t=0,1,\dots,i_1-1$, where $b_t=1$ if $t=p^{a-1}-i_1$ and $b_t=0$ otherwise.

Let $ M(p^a,i_1)$ be an $i_1\times i_1$ matrix defined by 
\begin{align}
	\label{M} M(p^a,i_1)={ \left( 
	\begin{array}{ccccc}
		(-1)^{i_0}+1 &0 &0 &\dots &0\\
		(-1)^{i_0}{i_0\choose 1} &(-1)^{i_0+1}+1 &0 &\dots &0\\
		(-1)^{i_0}{i_0\choose 2} &(-1)^{i_0+1}{i_0-1\choose 1}&(-1)^{i_0+2}+1&\dots &0\\
		\vdots &\vdots &\vdots &\ddots &\vdots\\
		(-1)^{i_0}{i_0\choose i_1-1}&(-1)^{i_0+1}{i_0-1\choose i_1-2}&(-1)^{i_0+2}{i_0-2\choose i_1-3}& \dots&(-1)^{i_0+i_1-1}+1 
	\end{array}
	\right)} .
\end{align}
For a matrix $V$, denote by $V^T$ the transpose of $V$. Then (\ref{setupH}) forms a matrix equation 
\begin{align}
	\label{H-eq} M(p^a,i_1)\boldsymbol{x}+( {\boldsymbol{x}} ^{p^{{s}/{2}}}-\boldsymbol{x})=\boldsymbol{b}, 
\end{align}
where $\boldsymbol{x}:=(x_1,x_2,\dots,x_{i_1})^{\rm T}$, $\boldsymbol{x}^{p^{{s}/{2}}}:=(x_1^{p^{{s}/{2}}},x_2^{p^{{s}/{2}}},\dots,x_{i_1}^{p^{{s}/{2}}})^{\rm T}$, and $\boldsymbol{b}:= (b_1,b_2,\dots,b_{i_1})^{\rm T}$ is a zero vector   except for  the case  $i_1\geq \frac{p^{a-1}-1}{2}$  where, for each $1\leq i\leq i_1$, ${b}_i$  is defined by
\begin{align}\label{bi}
	b_i=\begin{cases}
		1& \text{ if } i= p^{a-1}-i_1+1, \\
		0~~ & \text{ otherwise} .
	\end{cases}
\end{align}

Therefore, the cyclic code $C$ in (\ref{cyclic}) is Hermitian self-dual if and only if the matrix equation (\ref{H-eq}) has a solution in $\mathbb{F}_{p^s}^{i_1}$. Moreover, the number of Hermitian self-dual cyclic codes of length $p^a$ over ${\rm GR}(p^2,s)$ with first torsion degree $i_1$ equals the number of  solutions of (\ref{H-eq}) in $\mathbb{F}_{p^s}^{i_1}$.

From (\ref{M}), we observe that $M(p^a,i_1)$ has the following properties. For $1\leq j, i\leq i_1$, let $m_{ij}$ denote the entry  in the $i$th row and $j$th column of $M(p^a,i_1)$.
Then,  for integers $1\leq k\leq j\leq i\leq i_1$, we have \[m_{ij}=(-1)^{i_0+j-1}{i_0-j+1 \choose i-j},\]
 and hence, 
\begin{align}
	\label{aijk} m_{ij}m_{jk} 
	&=(-1)^{i_0+j-1}{i-k\choose j-k}m_{ik}. 
\end{align}
\begin{lemma}
	\label{ab=0} Let $i$ be an integer such that $1\leq i< i_1$. Then
	\[\sum_{j=1}^{i} m_{{i+1},j}b_j =0 \text{ in } \mathbb{F}_{p^s}.\]
\end{lemma}
\begin{proof}
If $i_1< \frac{p^{a-1}-1}{2}$, then $b_i=0$ for all $1\leq i\leq i_1$, and hence, the result follows.  

Assume that $i_1\geq \frac{p^{a-1}-1}{2}$. If $i< \frac{p^{a-1}+1}{2}$, then $b_j=0$ for all $1\leq j\leq i$, and hence, $\sum_{j=1}^{i} m_{{i+1},j}b_j =0 \text{ in }\mathbb{F}_{p^s} $.
Assume that $i\geq \frac{p^{a-1}+1}{2}$. Then, by (\ref{bi}),  we have   
	\begin{align*}
		\sum_{j=1}^{i} m_{{i+1},j}b_j&=m_{{i+1},p^{a-1}-i_1+1}\\
		&=(-1)^{i_0+p^{a-1}-i_1}{p^{a}-p^{a-1} \choose i_0+1-i}\\
		& =0 \text{ in } \mathbb{F}_{p^s} 
	\end{align*}
	by Lucas's Theorem (see \cite[Theorem 26]{BQ2003}) and the fact that $1\leq i_0-i<p^{a-1}$. 
\end{proof}

  In order to  determine the number of solutions   of (\ref{H-eq}) in $\mathbb{F}_{p^s}^{i_1}$, we recall two maps which are important tools.

\begin{enumerate}
\item The trace map ${\rm Tr}: \mathbb{F}_{p^s} \to \mathbb{F}_{p^{{s}/{2}}}$   defined by $\alpha \mapsto \alpha+\alpha^{p^{{s}/{2}}}$ for all $\alpha \in \mathbb{F}_{p^{s}}$. 
\item The map $ \Psi: \mathbb{F}_{p^s}\to \mathbb{F}_{p^s}$   defined by $ \Psi(\alpha)={\alpha}^{p^{{s}/{2}}}-\alpha$ for all $\alpha \in \mathbb{F}_{p^{s}}$. 
\end{enumerate}

It is well know that ${\rm Tr}$ is $\mathbb{F}_{p^{s/2}}$-linear and it is not difficult  to see that $ \Psi$ is also $\mathbb{F}_{p^{s/2}}$-linear.  If $p=2$, then  ${\rm Tr}=\Psi$. 
Moreover, we have the following properties.

\begin{lemma}\label{lemma1} For each  $\alpha\in\mathbb{F}_{p^{s}} $,  $ \Psi(\alpha)=0$ if and only if $\alpha\in\mathbb{F}_{p^{s/2}}$. 
Moreover, 	
\begin{align}
		\label{PsiTr}  \Psi\circ {\rm Tr}\equiv 0\equiv {\rm Tr}\circ  \Psi .
	\end{align}
\end{lemma}
\begin{proof}  It follows immediately from the definitions.
\end{proof}

\begin{lemma}\label{lemma2}  The following statements hold.
\begin{enumerate}[$i)$]
\item For each $a\in  \Psi( \mathbb{F}_{p^{s}})$, $|\Psi^{-1}({a})|=p^{{s}/{2}}$.
\item For each $a\in{\rm Tr}( \mathbb{F}_{p^{s}})$, $|{\rm Tr}^{-1}({a})|=p^{{s}/{2}}$.
\end{enumerate}
\end{lemma}
\begin{proof}   
We note that, for each $a\in  \Psi( \mathbb{F}_{p^{s}})$ and $b\in \mathbb{F}_{p^{s}}$, $b\in \Psi^{-1}({a})$ if and only if $b +\ker(\Psi)=\Psi^{-1}({a})$.  Since  $\ker(\Psi)=\mathbb{F}_{p^{s/2}}$, the statement $i)$ follows.

Similarly,  for each $a\in   {\rm Tr}( \mathbb{F}_{p^{s}})$ and $b\in \mathbb{F}_{p^{s}}$, $b\in {\rm Tr}^{-1}({a})$ if and only if $ b +\ker({\rm Tr})= {\rm Tr}^{-1}({a})$.  Since {\rm Tr} is a surjective  $\mathbb{F}_{p^{s/2}}$-linear map from $\mathbb{F}_{p^{s}}$  to $\mathbb{F}_{p^{s/2}}$,     we have  \[|\ker({\rm Tr})|= \frac{|\mathbb{F}_{p^{s}}|}{|\mathbb{F}_{p^{s/2}}|}=p^{s/2},\]
and hence, $ii)$ follows.
\end{proof}

\begin{proposition}
	\label{DualH-torsion} Let $s$ be an even positive integer and let $i_1$ be a positive integer such that $i_1\leq p^{a-1}$. Then the number of  solutions of (\ref{H-eq}) in $\mathbb{F}_{p^s}^{i_1}$ is
	\[p^{{si_1}/{2}}.\]
\end{proposition}
\begin{proof}

From  \cite{KLL2012},      $M(p^a,i_1)$ has  $4$ presentations depending on the parity of $p$ and  $i_1$.  We therefore separate the proof into  $4$ cases.

\noindent{\bf Case 1.} $p$ is odd and  $i_1=2\mu_1+1$ is odd.

 From  \cite{KLL2012}, the matrix $M(p^a,i_1)$ can be written as  
\begin{align}
		\label{MpI1odd} M(p^a,i_1)={  \left( 
		\begin{array}{cccccc}
			2&0&0&0&\cdots&0\\
			*&0&0&0&\cdots&0\\
			*&*&2&0&\cdots&0\\
			*&*&*&0&\cdots&0\\
			\vdots&\vdots&\vdots&\vdots&\ddots&\vdots\\
			*&*&*&*&\cdots&2\\
		\end{array}
		\right)},
	\end{align}
 where $*$'s denote entries of $M(p^a,i_1)$  defined in (\ref{M}).

From (\ref{H-eq}) and (\ref{MpI1odd}), we have 
	\begin{align}
		{\rm Tr}(x_{1})&=b_{1},\label{O0}
\end{align}
\begin{align}
		 \Psi(x_{2i})&=b_{2i}- \sum_{j=1}^{2i-1} m_{{2i},j}x_j, \label{O1}
\end{align}
and 
\begin{align}
		{\rm Tr}(x_{2i+1})&=b_{2i+1}- \sum_{j=1}^{2i} m_{{2i+1},j}x_j\label{O2} 
	\end{align}
	for all integers $1\leq i\leq \mu_1$.
	
	This implies that (\ref{H-eq}) has a solution if and only if the right hand sides of (\ref{O0}) and (\ref{O2}) are in ${\mathbb{F}_{p^{s/2}}}$ and the right hand side of (\ref{O1}) is in $ \Psi(\mathbb{F}_{p^s})$. In this case, we have 
	\begin{align*}
		x_1&\in {\rm Tr}^{-1}(b_1),
\end{align*}
\begin{align*}
		x_{2i} &\in  \Psi^{-1}\left(b_{2i}-\sum_{j=1}^{2i-1} m_{{2i},j}x_j\right), 
\end{align*}
and
\begin{align*}
		x_{2i+1}&\in {\rm Tr}^{-1}\left(b_{2i+1}- \sum_{j=1}^{2i} m_{{2i+1},j}x_j\right) 
	\end{align*}
	for all $1\leq i\leq \mu_1$.  Hence, by Lemma \ref{lemma2}, the number of solutions of (\ref{H-eq}) is $p^{{si_1}/{2}}$.
	
	By Lemma \ref{lemma1}, it suffices to show that the images under $ \Psi$ of the right hand sides of (\ref{O0}) and (\ref{O2}) are $0$ and the image under the trace map of the right hand side of (\ref{O1}) is $0$.
	
	From (\ref{O0}), we have $\Psi(b_1)=0$ since $b_1\in \{0,1\}\subseteq \mathbb{F}_{p^{s/2}}$. Let $1\leq i\leq \mu_1$ be an integer.  From (\ref{O1}), we have

	\begin{align}
		\label{ODDp-ODD} {\rm Tr}\left(b_{2i}-\sum_{j=1}^{2i-1} m_{{2i},j}x_j\right) &={\rm Tr}(b_{2i})- \sum_{j=1}^{2i-1} m_{{2i},j}{\rm Tr}(x_j)\notag\\
		&=0- \left(\sum_{j=1}^{i} m_{{2i},2j-1}{\rm Tr}(x_{2j-1})+ \sum_{j=1}^{i-1} m_{{2i},2j}{\rm Tr}(x_{2j})\right),\notag\\
		&~~~~ \text{ since } p^{a-1}-i_1+1 \text{ is odd and then } b_{2i}=0\text{ for all }i=1,2,\dots, \mu_1,\notag\\
		&=- \left(\sum_{j=1}^{i} m_{{2i},2j-1}{\rm Tr}(x_{2j-1})+ \sum_{j=1}^{i-1} m_{{2i},2j}\left(\Phi(x_{2j})+2x_{2j}\right)\right),\notag\\
		&~~~~ \text{ since } {\rm Tr}(\alpha)=\Phi(\alpha)+2\alpha \text{ for all } \alpha\in \mathbb{F}_{p^s},\notag\\		
		&=-\sum_{j=1}^{2i-1} m_{{2i},j}b_j +\sum_{j=1}^{2i-1}\sum_{k=1}^{j-1} m_{{2i},j}m_{j,k}x_k- \sum_{j=1}^{i-1}2m_{2i,2j}x_{2j}\notag\\
		&=0 + \sum_{k=1}^{2i-2} \left(\sum_{j=k+1 }^{2i-1} m_{{2i},j}m_{j,k}\right)x_k- \sum_{j=1}^{i-1}2m_{2i,2j}x_{2j}, \text{ by Lemma~\ref{ab=0}},\notag\\
		&= \sum_{k=1}^{2i-2} \left(\sum_{j=k+1 }^{2i-1} (-1)^{i_0+j-1}{2i-k \choose j-k} m_{2i,k}\right)x_k- \sum_{j=1}^{i-1}2m_{2i,2j}x_{2j}, \text{ by (\ref{aijk}),}\notag\\
		&= \sum_{k=1}^{2i-2} \left(m_{2i,k}\sum_{j=1 }^{2i-k-1} (-1)^{j+k-1}{2i-k \choose j} \right)x_k- \sum_{j=1}^{i-1}2m_{2i,2j}x_{2j}\notag\\
		&= \sum_{k=1}^{i} \left( m_{2i,2k-1} \sum_{j=1 }^{2i-2k} (-1)^{j}{2i-2k+1 \choose j} \right)x_{2k-1}\notag\\
		&~~~~ +\sum_{k=1}^{i-1} \left( m_{2i,2k} \sum_{j=1 }^{2i-2k-1} -(-1)^{j}{2i-2k \choose j} \right)x_{2k}- \sum_{j=1}^{i-1}2m_{2i,2j}x_{2j}\notag\\
		&= \sum_{k=1}^{i} \left(m_{2i,2k-1} (-1-(-1)^{2i-2k+1}) \right)x_{2k-1}\notag\\
		&~~~~ +\sum_{k=1}^{i-1} \left(-(-1-(-1)^{2i-2k})m_{2i,2k} \right)x_{2k}- \sum_{j=1}^{i-1}2m_{2i,2j}x_{2j}\notag\\
		&=0 \text{ in } \mathbb{F}_{p^{{s}/{2}}}. 
	\end{align}
	
	From (\ref{O2}), we have 
	\begin{align}
		\label{ODDp-ODD2} \Psi\left(b_{2i+1}-\sum_{j=1}^{2i} m_{{2i+1},j}x_j\right) &=\Psi(b_{2i+1})-\sum_{j=1}^{2i} m_{{2i+1},j}\Psi(x_j)\notag\\
		&=0-\left( \sum_{j=1}^{i} m_{{2i+1},2j-1}\Psi(x_{2j-1}) + \sum_{j=1}^{i} m_{{2i+1},2j}\Psi(x_{2j})\right),\notag\\
		&~~~~ \text{ since }b_{2i+1}\in \mathbb{F}_{p^\nu}\text{ for all }i=1,2,\dots, \mu_1,\notag\\
		&=- \left(\sum_{j=1}^{i} m_{{2i+1},2j-1}\left({\rm Tr}(x_{2j-1})-2x_{2j-1}\right)+ \sum_{j=1}^{i} m_{{2i+1},2j}\Psi\left(x_{2j}\right)\right),\notag\\
		&~~~~ \text{ since } {\Psi(\alpha)=\rm Tr}(\alpha)-2\alpha \text{ for all } \alpha\in \mathbb{F}_{p^s},\notag\\
		&=-\sum_{j=1}^{2i} m_{{2i+1},j}b_j +\sum_{j=1}^{2i}\sum_{k=1}^{j-1} m_{{2i+1},j}m_{j,k}x_k+ \sum_{j=1}^{i}2m_{2i+1,2j-1}x_{2j-1}\notag\\
		&=0 + \sum_{k=1}^{2i-1} \left(\sum_{j=k+1 }^{2i} m_{{2i+1},j}m_{j,k}\right)x_k+ \sum_{j=1}^{i}2m_{2i+1,2j-1}x_{2j-1},\notag \\
		&~~~~\text{ by Lemma~\ref{ab=0}},\notag\\
		&= \sum_{k=1}^{2i-1} \left(\sum_{j=k+1 }^{2i} (-1)^{i_0+j-1}{2i-k+1 \choose j-k} m_{2i+1,k}\right)x_k \notag\\
		&~~~~+ \sum_{j=1}^{i}2m_{2i+1,2j-1}x_{2j-1}, \text{ by (\ref{aijk}),}\notag\\
		&= \sum_{k=1}^{i} \left( m_{2i+1,2k-1} \sum_{j=1 }^{2i-2k-1} (-1)^{j}{2i-2k +2\choose j} \right)x_{2k-1}\notag\\
		&~~~~ + \sum_{k=1}^{i} \left(m_{2i,2k} \sum_{j=1 }^{2i-2k} -(-1)^{j}{2i-2k +1\choose j} \right)x_{2k}+ \sum_{j=1}^{i}2m_{2i+1,2j-1}x_{2j-1}\notag\\
		&= \sum_{k=1}^{i} \left( (-1-(-1)^{2i-2k +2}) m_{2i+1,2k-1} \right)x_{2k-1}\notag\\
		&~~~~ + \sum_{k=1}^{i} \left( -(-1-(-1)^{2i-2k +1}) m_{2i,2k} \right)x_{2k}+ \sum_{j=1}^{i}2m_{2i+1,2j-1}x_{2j-1}\notag\\
		&=0 \text{ in }\mathbb{F}_{p^{s/2}}. 
	\end{align}

This case is completed.

\vskip1em
\noindent{\bf Case 2.} $p$ is odd and  $i_1=2\mu_1$ is even. 

From \cite{KLL2012}, we have
	\begin{align}
		\label{MpI1even} M(p^a,i_1)= {  \left( 
		\begin{array}{cccccc}
			0&0&0&0&\cdots&0\\
			*&2&0&0&\cdots&0\\
			*&*&0&0&\cdots&0\\
			*&*&*&2&\cdots&0\\
			\vdots&\vdots&\vdots&\vdots&\ddots&\vdots\\
			*&*&*&*&\cdots&2\\
		\end{array}
		\right)}, 
	\end{align}
   where $*$'s denote entries of $M(p^a,i_1)$  defined in (\ref{M}).

	From (\ref{H-eq}) and (\ref{MpI1even}), we have 
	\begin{align}
		\Psi(x_{2i-1})&=b_{2i-1}- \sum_{j=1}^{2i-2} m_{{2i-1},j}x_j \label{OE1} 
\end{align}
and 
\begin{align}
		{\rm Tr}(x_{2i})&=b_{2i}- \sum_{j=1}^{2i-1} m_{{2i},j}x_j\label{OE2} 
	\end{align}
	for all integers $1\leq i\leq \mu_1$.
	
	This implies that (\ref{H-eq}) has a solution if and only if the right hand side of (\ref{OE1}) is in $\Psi(\mathbb{F}_{p^s})$ and the right hand side of (\ref{OE2}) is in ${\mathbb{F}_{p^{s/2}}}$.
In this case, by Lemma \ref{lemma2}, the number of  solutions of (\ref{H-eq}) is $p^{{{si_1}/{2}}}$.
	
	By Lemma \ref{lemma1}, it is sufficient to show that the image under the trace map of the right hand side of (\ref{OE1})   and the image under $\Psi$ of the right hand side of (\ref{OE2}) are $0$. 
Using computations similar to those in (\ref{ODDp-ODD}) and (\ref{ODDp-ODD2}), the desired properties can be concluded.

%
%

\vskip1em

\noindent{\bf Case 3.} $p=2$   and $i_1=2\mu_1+1$ is odd. 
	
From \cite{KLL2012}, we have
	\begin{align}
		 M(p^a,i_1)
={   \left( 
		\begin{array}{cccccc}
			0&0&0&0&\cdots&0\\
			1&0&0&0&\cdots&0\\
			*&0&0&0&\cdots&0\\
			*&*&1&0&\cdots&0\\
			\vdots&\vdots&\vdots&\vdots&\ddots&\vdots\\
			*&*&*&*&\cdots&0\\
		\end{array}
		\right) }, \label{M2I1odd}
\end{align}   
     where $*$'s denote entries of $M(p^a,i_1)$  defined in (\ref{M}).

From  (\ref{H-eq}) and (\ref{M2I1odd}), we conclude that 
	\begin{align}
		{\rm Tr}(x_1)&=b_1, \label{e0}
	\end{align}
	\begin{align}
		{\rm Tr}(x_{2i})&=b_{2i}+ \sum_{j=1}^{2i-1} m_{{2i},j}x_j, \label{e1}
		\end{align}
  and 
		\begin{align}
		{\rm Tr}(x_{2i+1})&=b_{2i+1}+ \sum_{j=1}^{2i-1} m_{{2i+1},j}x_j\label{e2} 
	\end{align}
	for all integers $1\leq i\leq \mu_1$.
	
Similar to Cases 1 and 2, we need to show that  the right hand sides of (\ref{e0}), (\ref{e1}), and (\ref{e2}) are in ${\mathbb{F}_{2^{{s}/{2}}}}$, or equivalently, the images under the trace map of the right hand sides of (\ref{e0}), (\ref{e1}), and (\ref{e2}) are $0$. Clearly, the right hand side of (\ref{e0}) is $b_1\in \{0,1\}\subseteq {\mathbb{F}_{2^{s}/{2}}}$ and ${\rm Tr}(b_1)=0$. Let $1\leq i\leq \mu_1$ be an integer. 
	
	From (\ref{e1}), we have 
	\begin{align}
		\label{ODD} {\rm Tr}\left(b_{2i}+\sum_{j=1}^{2i-1} m_{{2i},j}x_j\right) &=\sum_{j=1}^{2i-1} m_{{2i},j}{\rm Tr}(x_j)\notag\\
		&=\sum_{j=1}^{2i-1} m_{{2i},j}\left(b_j+ \sum_{k=1}^{j-1} m_{j,k}x_k \right)\notag\\
		&=\sum_{j=1}^{2i-1} m_{{2i},j}b_j +\sum_{j=1}^{2i-1}\sum_{k=1}^{j-1} m_{{2i},j}m_{j,k}x_k\notag\\
		&=0 + \sum_{k=1}^{2i-2} \left(\sum_{j=k+1 }^{2i-1} m_{{2i},j}m_{j,k}\right)x_k, \text{ by Lemma~\ref{ab=0},}\notag\\
		&= \sum_{k=1}^{2i-2} \left(\sum_{j=k+1 }^{2i-1} {2i-k \choose j-k} m_{2i,k}\right)x_k, \text{ by (\ref{aijk}),}\notag\\		
		&=\sum_{k=1}^{2i-2} \left(m_{2i,k}\sum_{j=1 }^{2i-k-1} {2i-k \choose j} \right)x_k\notag\\
		&=\sum_{k=1}^{2i-2} \left(m_{2i,k}(2^{2i-k} -2)\right)x_k\notag\\
		&=0 \text{ in }\mathbb{F}_{2^{s/2}}. 
	\end{align}
	
	Applying a similar  computation to (\ref{e2}) yields 
	\begin{align*}
		{\rm Tr}\left(b_{2i+1}+\sum_{j=1}^{2i} m_{{2i+1},j}x_j\right) 
		&=0 \text{ in }\mathbb{F}_{2^{{s}/{2}}}. 
	\end{align*}


\noindent{\bf Case 4.} $p=2$  and $i_1=2\mu_1+2$ is even. 
	
	Form \cite{KLL2012},  we have 
\begin{align}
		 M(p^a,i_1)=
 { \left( 
		\begin{array}{ccccccc}
			0&0&0&0&0&\cdots&0\\
			0&0&0&0&0&\cdots&0\\
			*&1&0&0&0&\cdots&0\\
			*&*&0&0&0&\cdots&0\\
			*&*&*&1&0&\cdots&0\\
			\vdots&\vdots&\vdots&\vdots&\vdots&\ddots&\vdots\\
			*&*&*&*&*&\cdots&0\\
		\end{array}
		\right)},   \label{M2I1even} 
\end{align}
   where $*$'s denote entries of $M(p^a,i_1)$  defined in (\ref{M}).
	
	From  (\ref{H-eq}) and (\ref{M2I1even}), it follows that 
	\begin{align}
		{\rm Tr}(x_1)&=b_1 \label{E0},
		\end{align}
	\begin{align}
		{\rm Tr}(x_2)&=b_2 \label{E00},
		\end{align}
	\begin{align}
		{\rm Tr}(x_{2i+1})&=b_{2i+1}+ \sum_{j=1}^{2i} m_{{2i+1},j}x_j,  \label{EE1}
		\end{align}
and 
		\begin{align}
		{\rm Tr}(x_{2i+2})&=b_{2i+2}+ \sum_{j=1}^{2i} m_{{2i+2},j}x_j\label{EE2} 
	\end{align}
	for all integers $1\leq i\leq \mu_1$. 

Similar to the previous cases, we need to show that  the right hand sides  of (\ref{E0}), (\ref{E00}), (\ref{EE1}), and (\ref{EE2}) are in ${\mathbb{F}_{2^{s/2}}}$. Equivalently,  the images under the trace map of the right hand sides of (\ref{E0}), (\ref{E00}), (\ref{EE1}), and (\ref{EE2}) are $0$. Clearly, the traces of the right hand sides of (\ref{E0}) and (\ref{E00}) are $0$ since $b_1, b_2\in \{0,1\}\subseteq {\mathbb{F}_{2^{s/2}}}$. 

For each integer $1\leq i\leq \mu_1$, using computations similar to those in (\ref{ODD}), the trace of the right hand sides of (\ref{EE1})  and (\ref{EE2}) become $0 \text{ in }\mathbb{F}_{2^{s/2}} $.

	
	The proof is now completed.  
\end{proof}

The next theorem follows immediately from Propositions~\ref{Structure-cyclic} and~\ref{DualH-torsion}. 
\begin{theorem}
	\label{NHC} Let $p$ be a prime and let $s,a$ be positive integers such that $s$ is even. Then
	\[N_{\rm H}({\rm GR}(p^2,s),p^a)= \sum_{i_1=0}^{p^{a-1}} p^{{si_1}/{2}}= \frac{p^{{{s(p^{a-1}+1)}/{2}}}-1}{p^{{{s}/{2}}}-1}.\]
\end{theorem}

\section{Euclidean Self-Dual Cyclic Codes of Arbitrary Length over ${\rm GR}(p^2,s)$}\label{sec4}

In this section, we focus on cyclic codes of any length $n$ over  ${\rm GR}(p^2,s)$. We generalize   the  decomposition in \cite{DL2006}  (see also \cite{B2003} and \cite{DP2007}) to this case.  Combining with the results in \cite{KLL2008}, \cite{KLL2012}, and  in the previous section, we characterize and enumerate Euclidean self-dual cyclic codes of  any length $n$ over  ${\rm GR}(p^2,s)$. Some corrections to    \cite{B2003} and \cite{DL2006}  are also provided.

Write  $n=mp^a$, where $p\nmid m$ and $a\geq 0$.  
Let $R_{p^2}(u,s):={\rm GR}(p^2,s)[u]/\langle u^{p^a}-1 \rangle$.  Let $\widetilde{~~}$ be an involution on $R_{p^2}(u,s)$ that fixes ${\rm GR}(p^2,s)$ and that maps $u^i$ to $u^{-i\,({\rm mod}\,p^a)}$ for all $0\leq i< p^a$, and extend $\widetilde{~~}$ naturally to ${\rm GR}(p^2,s\nu)[u]/\langle u^{p^a}-1 \rangle$ for all positive integers $\nu$.  It is not difficult to verify that  the map  $\Phi: (R_{p^2}(u,s))^m\to {\rm GR}(p^2,s)[X]/\langle X^{mp^a}-1\rangle$ defined by 
\[\left(\sum_{j=0}^{p^a-1} c_{0,j}  u^{j}, \sum_{j=0}^{p^a-1} c_{1,j}  u^{j},\dots, \sum_{j=0}^{p^a-1} c_{m-1,j}  u^{j}\right)\mapsto \sum_{i=0}^{m-1} \sum_{j=0}^{p^a-1} c_{i,j}  X^{i+jm} \]
is a ${\rm GR}(p^2,s)$-module isomorphism. 

The following lemma is an obvious generalization of \cite[Lemma 5.1]{DL2006}.
\begin{lemma}\label{ortho} Let  $\boldsymbol{d}=(d_0,d_1,\dots,d_{m-1})$ and $\boldsymbol{d}^\prime=(d_0^\prime,d_1^\prime,\dots,d_{m-1}^\prime)$ be elements in $(R_{p^2}(u,s))^m$. 
Then $[\boldsymbol{d},\boldsymbol{d}^\prime]:=\sum_{i=0}^{m-1}d_i\widetilde{d_i^\prime}=0$ if and only if  $\langle X^{mi}\Phi(\boldsymbol{d}), \Phi(\boldsymbol{d}^\prime)\rangle_{\rm E}$ for all $0\leq i\leq p^a-1$.
\end{lemma}
The lemma implies that, for cyclic codes $C_1$ and $C_2$ of  length $n$ over  ${\rm GR}(p^2,s)$, $C_2=C_1^{\perp_{\rm E}}$ if and only if   $\Phi^{-1}(C_2)$ is the dual of $\Phi^{-1}(C_1)$ under the form $[\cdot,\cdot]$. In particular, $C_1$ is Euclidean self-dual if and only if $\Phi^{-1}(C_1)$ self-dual under the form $[\cdot,\cdot]$.

\subsection{Decomposition}
 We generalize the Discrete Fourier Transform decomposition for cyclic codes over $\mathbb{Z}_4$  in \cite{DL2006}    to cyclic codes over ${\rm GR}(p^2,s)$ as follows.
 
Let $M$ be the multiplicative order of $p^s$ modulo $m$ and let $\zeta$ denote a primitive $m$th root of unity in  ${\rm GR}(p^2,sM)$.
The Discrete Fourier Transform  of 
\[c(X)= \sum_{i=0}^{m-1}\sum_{j=0}^{p^a-1} c_{i,j}X^{i+jm}\in {\rm GR}(p^2,s)[X]/\langle X^{mp^a}-1\rangle\]
 is the vector
\[(\widehat{c}_0, \widehat{c}_1,\dots, \widehat{c}_{m-1})\in R_{p^2}(u,sM)^m\]
with
\[\widehat{c}_h=\sum_{i=0}^{m-1}\sum_{j=0}^{p^a-1} c_{i,j}  u^{m^\prime i+j} \zeta^{hi} \]
for all $0\leq h\leq m-1$, where $mm^\prime\equiv 1 \,({\rm mod}\,p^a)$.

Define the Mattson-Solomon polynomial of $c(X)$ to be 
\[\widehat{c}(Z)= \sum_{h=0}^{m-1} \widehat{c}_{m-h \,({\rm mod}\,m) } Z^h. \] 
Then the following lemma can be obtained as an extension of  \cite[Lemma 3.1]{DL2006}. 
\begin{lemma}[Inversion formula]  Let $c(X) \in {\rm GR}(p^2,s)[X]/\langle X^{mp^a}-1\rangle$ with $\widehat{c}(Z)$ its Mattson-Solomon polynomial as defined above. Then 
\[	c(X)=\Phi\left( (1,u^{-m^\prime},u^{-2m^\prime},\dots, u^{-(m-1)m^\prime})\star\frac{1}{m}(\widehat{c}(1), \widehat{c}(\zeta),\dots, \widehat{c}(\zeta^{m-1})) \right),\]
where $\star$ indicates  the componentwise multiplication. 
\end{lemma}

   

For $0\leq h\leq m-1$, denote by $S_{p^s}(h)$ the $p^s$-cyclotomic coset of $h$ modulo $m$, {\em i.e.}, 
\[S_{p^s}(h)=\{hp^{is}({\rm mod}\,m)\mid i=0,1,\dots\}, \]
and denote by $m_h$ the size of $S_{p^s}(h)$. Then $m_h$ is the multiplicative order of $p^s$ modulo ${\rm ord}(h)$, where  ${\rm ord}(h)$ denotes the additive order of $h$ modulo $m$. Since ${\rm ord}(h)={\rm ord}(-h)$, we have $m_h=m_{-h}$ for all $0\leq h\leq m-1$.

The $p^s$-cyclotomic coset $S_{p^s}(h)$ is said to be {\em self-inverse} if $S_{p^s}(-h)=S_{p^s}(h)$, or equivalently, $-h \in S_{p^s}(h)$. In this case,   the size $m_h$ of $S_{p^s}(h)$ is $1$ or even (see \cite[Remark 2.6]{JLLX2013}). Moreover, we have $-h=h$ if $m_h=1$, and  $-h=p^{{sm_h/2}}h$ otherwise.  We note that $S_{p^s}(0)$ and $S_{p^s}(\frac{m}{2})$ (if $m$ is even) are self-inverse.  

\begin{remark} We have the following observations for the coefficients of the Discrete Fourier Transform.
\begin{enumerate}[$i)$]
\item If  $S_{p^s}(h)$ is self-inverse of size $1$, then 
  	\begin{align}\label{self-inv-sing}
		\widehat{c}_{m-h}= \widehat{c}_h.
	\end{align}
\item If $S_{p^s}(h)$ is self-inverse of size $2e$, then  $-h=p^{se}h$ and 
	\begin{align}\label{self-inv}
		\widehat{c}_{m-h}&=\sum_{i=0}^{m-1}\sum_{j=0}^{p^a-1} c_{i,j}  u^{m^\prime i+j} \zeta^{-hi}\notag\\
		                 &=\sum_{i=0}^{m-1}\sum_{j=0}^{p^a-1} c_{i,j}  u^{m^\prime i+j} \zeta^{p^{se}i}\notag\\
		                 &=	\overline{\widehat{c}_{h}},
	\end{align}
	where $\bar{~}$ is a natural extension of (\ref{aut}), \textit{i.e.}, $\bar{~}$  fixes $u$ and maps $a+pb $ to $a^{p^{se}}+pb^{p^{se}}$.
\end{enumerate}
\end{remark}

Let $I_0=\{0\}$ if $m$ is odd and $I_0=\{0,\frac{m}{2}\}$ if $m$ is even.
Let $I_1$ be the union of all self-inverse $p^s$-cyclotomic cosets modulo $m$ excluding $I_0$ and set  $I_2=\{0,1,\dots,m-1\}\setminus (I_0\cup I_1)$.   The set $I_2$ is the union of pairs of  $p^s$-cyclotomic cosets of the form $S_{p^s}(h)\cup S_{p^s}(-h)$, where $h\notin I_0\cup I_1$. Let  $J_0$, $J_1$, and $J_2$ be  complete sets of representatives of $p^s$-cyclotomic cosets in $I_0$, $I_1$, and $I_2$, respectively. Without loss of generality,  we assume that $J_2$ is chosen such that $h\in J_2$ if and only if $-h\in J_2$.

The following lemma is a straightforward extension of \cite[Theorem 3.2 and Corollary 3.3]{DL2006}.
\begin{lemma} \label{isoiso}
	The ring ${\rm GR}(p^2,s)[X]/\langle X^{mp^a}-1\rangle$ is isomorphic to \[\prod_{h\in J_0\cup J_1\cup J_2} R_{p^2}(u,sm_h)\] via the ring isomorphism \[c(X)\mapsto (\widehat{c}_{h})_{h\in J_0\cup J_1\cup J_2}.\]
	
	If $C$ is a cyclic code of length $n=mp^a$ over ${\rm GR}(p^2,s)$, then  $C$ is isomorphic to 
\begin{align}
\label{DEC}
\prod_{h\in J_0\cup J_1\cup J_2} C_h,
\end{align}
 where $C_h$ is an ideal in $R_{p^2}(u,sm_h)$ for all  $h\in J_0\cup J_1\cup J_2$.
\end{lemma}

\subsection{Euclidean Self-Dual Cyclic Codes}
In this subsection, we consider the Euclidean dual of cyclic codes of length $n=mp^a$ over ${\rm GR}(p^2,s)$. The characterization  and enumeration of  Euclidean self-dual  cyclic codes of length $n$ are established, where $p$ is a prime such that $p\nmid m$ and $a\geq 0$.

\begin{lemma}\label{lem-inner} Let $c(X) $ and $c^\prime(X) $  be polynomials in $ {\rm GR}(p^2,s)[X]/\langle X^{mp^a}-1\rangle$ with    Mattson-Solomon polynomials 	$\widehat{c}(Z)= \sum_{h=0}^{m-1} \widehat{c}_{m-h \,({\rm mod}\,m)}  Z^h$ and $\widehat{c^\prime}(Z)= \sum_{h=0}^{m-1} \widehat{c^\prime}_{m-h \,({\rm mod}\,m) } Z^h$, respectively. Let $(d_0,d_1,\dots,d_{m-1})=\Phi^{-1}(c(X))$ and $(d_0^\prime,d_1^\prime,\dots,d_{m-1}^\prime)=\Phi^{-1}(c^\prime(X))$. Then 
\begin{align}\label{eq-ortho}
\sum_{i=0}^{m-1}d_i\widetilde{d_i^\prime}
=\frac{1}{m}\left(\sum_{j\in I_0} \widehat{c}_j \widetilde{ {\widehat{c^\prime}_{j}}} 
+\sum_{h\in I_1}  \widehat{c}_h\widetilde{ \overline{\widehat{c^\prime}_{h}}}   
+\sum_{k\in I_2}  \widehat{c}_k \widetilde{\widehat{c^\prime}}_{m-k \,({\rm mod}\,m)}  \right),\end{align}	
where $\bar{~}$ is defined as in (\ref{aut}) in the appropriate Galois extension.
\end{lemma}
\begin{proof} 
Using   computations similar to those in  \cite[Equation (21)]{DL2006}, we have 
\begin{align*}
\sum_{i=0}^{m-1}d_i\widetilde{d_i^\prime}
&=\frac{1}{m}\left(\sum_{i=0}^{m-1} \widehat{c}_i \widetilde{\widehat{c^\prime}}_{m-i\,({\rm mod}\,m)} \right)\\
&=\frac{1}{m}\left(\sum_{j\in I_0} \widehat{c}_j \widetilde{\widehat{c^\prime}}_{m-j\,({\rm mod}\,m)}  
+\sum_{h\in I_1}  \widehat{c}_h \widetilde{\widehat{c^\prime}}_{m-h\,({\rm mod}\,m)}   
+\sum_{k\in I_2}  \widehat{c}_k \widetilde{\widehat{c^\prime}}_{m-k\,({\rm mod}\,m)}   \right).
\end{align*}
By  (\ref{self-inv-sing}) and (\ref{self-inv}),  we conclude the lemma.
\end{proof}

 Based on the isomorphism defined in Lemma \ref{isoiso}, we determine the Euclidean duals  of cyclic codes over ${\rm GR}(p^2,s)$ as follows.
\begin{proposition}\label{char}
Let $C$ be a cyclic code of length $n=mp^a$ over ${\rm GR}(p^2,s)$ decomposed as in (\ref{DEC}), {\em i.e.},
\begin{align}
\label{CC}
C\cong \prod_{j\in J_0} C_j \times \prod_{h \in   J_1 } C_h \times \prod_{k\in  J_2} C_k.
\end{align}
Then 
\begin{align}\label{DC}
C^{\perp_{\rm E}}\cong \prod_{j\in J_0} C_j^{\perp_{\rm E}} \times \prod_{h \in    J_1 } C_h^{\perp_{\rm H}} \times \prod_{k\in  J_2} C^{\perp_{\rm E}}_{m-k\,({\rm mod}\,m)}.
\end{align}

Moreover, $C$ is Euclidean self-dual if and only if $C_j$ is Euclidean self-dual  for all $j\in J_0$, $C_h$ is Hermitian self-dual  for all $h\in J_1$, and  $C_k = C^{\perp_{\rm E}}_{m-k\,({\rm mod}\,m)}$    for all $k\in J_2$.
\end{proposition}
\begin{proof} 
	Let $D$ be a cyclic code of length $n=mp^a$ over ${\rm GR}(p^2,s)$ such that   \[D\cong \prod_{j\in J_0} C_j^{\perp_{\rm E}} \times \prod_{h \in    J_1 } C_h^{\perp_{\rm H}} \times \prod_{k\in  J_2} C^{\perp_{\rm E}}_{m-k\,({\rm mod}\,m)}.\]

Consider $m=1$ in Lemma \ref{ortho}, we have the following facts.
	For each positive integer $\nu$ and  $a,b \in R_{p^2}(u,s\nu)$,       we have $\langle a,b\rangle _{\rm E}=0$  if  $a\widetilde{b}=0$, and
 $\langle a,b\rangle _{\rm H}=0$  if  $a\widetilde{\overline{b}}=0$.
	Therefore, by Lemmas~\ref{ortho} and \ref{lem-inner}, we have   $D\subseteq C^{\perp_{\rm E}}$. The equality follows from their cardinalities.
	
The second part is clear.
\end{proof}

 Since $J_2$ has been  chosen such that $h\in J_2$ if and only if $-h\in J_2$, we can write $J_2$ as a disjoint union $J_2=J_2^\prime\cap J_2^{\prime\prime}$, where  $J_2^\prime\subseteq J_2$ and $J_2^{\prime\prime}=\{-h\mid h\in J_2^\prime\}$.

\begin{corollary}\label{comp} 
	The  number of Euclidean self-dual codes of length $n=mp^a$ over ${\rm GR}(p^2,s)$ is 
	\[  \left(N_{\rm E}({\rm GR}(p^2,s),p^a ) \right)^{\delta(m)}\times \prod_{h\in J_1} N_{\rm H}( {\rm GR}(p^2,sm_h),p^a ) \times \prod_{k\in J_2^\prime} N({\rm GR}(p^2,sm_k),p^a  ) ,\]
	where 
	\[\delta(m)=\begin{cases}1 &\text{ if $m$ is odd},\\
	                         2 &\text{ if $m$ is even},
	            \end{cases}
	\]
and the empty product is regarded as $1$.
\end{corollary}
\begin{proof}
From Proposition \ref{char}, a code  $C$ decomposed as in (\ref{CC}) is Euclidean self-dual if and only if  $C_j$ is Euclidean self-dual  for all $j\in J_0$, $C_h$ is Hermitian self-dual  for all $h\in J_1$, and  $C_k = C^{\perp_{\rm E}}_{m-k\,({\rm mod}\,m)}$    for all $k\in J_2$.

The number of Euclidean (resp., Hermitian) self-dual cyclic codes of length $p^a$ corresponding to $J_0$ (resp., $J_1$) is $\prod_{j\in J_0} N_{\rm E}({\rm GR}(p^2,s),p^a ) $ (resp., $\prod_{h\in J_1} N_{\rm H}( {\rm GR}(p^2,sm_h),p^a )$). The number of  choices of cyclic codes of length $p^a$ corresponding to $J_2^\prime$ is $\prod_{k\in J_2^\prime} N({\rm GR}(p^2,sm_k),p^a  )$.  Since  $C_k = C^{\perp_{\rm E}}_{m-k\,({\rm mod}\,m)}$    for all $k\in J_2$,   there is only one possibility for   codes corresponding to $J_2^{\prime\prime}$.

Therefore, the number of Euclidean self-dual codes of length $n=mp^a$ over ${\rm GR}(p^2,s)$ is 
\[  \prod_{j\in J_0} N_{\rm E}({\rm GR}(p^2,s),p^a )   \times \prod_{h\in J_1} N_{\rm H}( {\rm GR}(p^2,sm_h),p^a ) \times \prod_{k\in J_2^\prime} N({\rm GR}(p^2,sm_k),p^a  ). \]
Since $|J_0|=1$ if $m$ is odd and $|J_0|=2$ if $m$ is even, the result follows.
\end{proof}
We note that,  for each positive integer $t$, $\{0\}$, $p{\rm GR}(p^2,st)$, and  ${\rm GR}(p^2,st)$ are  all cyclic codes of length $1$ over ${\rm GR}(p^2,st)$ and  $p{\rm GR}(p^2,st)$ is the only Euclidean self-dual cyclic code. Hence, $N_{\rm E}({\rm GR}(p^2,st),1 )=1$, and $   N({\rm GR}(p^2,st),1  )=3$ for all positive integers $t$.  If $st$ is even, then $p{\rm GR}(p^2,st)$  is the only Hermitian self-dual cyclic code, and hence, $N_{\rm H}( {\rm GR}(p^2,st),1 ) =1$. 

 For $a\geq 1$, the numbers  $N_{\rm E}$, $N_{\rm H}$, and $N$ have been determined in Corollary \ref{N-cyclic2}, Proposition \ref{NEC}, and Theorem \ref{NHC}, respectively. 

Therefore, the number in Corollary \ref{comp} is completely determined. 

\begin{corollary}\label{pm} Let $m$ be a positive  integer be such that $p\nmid m$. 
	Then 
	$ N_{\rm E}({\rm GR}(p^2,s),mp )= 1$ if and only if $m=1$ and $p=2$.
\end{corollary}
\begin{proof}
Assume that $m>1$ or $p$ is odd.
\\{\bf Case 1.} $m>1$. 

If $J_2$ is not empty, then, by Corollaries \ref{N-cyclic2} and  \ref{comp},   $N_{\rm E}({\rm GR}(p^2,s),mp)\geq N( {\rm GR}(p^2,sm_h),p)>1$ for all $h\in J_2$. Assume that $J_2$ is empty. Since $m>1$, $J_1$ is not empty. Then, by Theorem \ref{NHC} and Corollary \ref{comp}, $N_{\rm E}( {\rm GR}(p^2,s),mp)\geq N_{\rm H}( {\rm GR}(p^2,sm_h),p)>1$ for all $h\in J_1$.
\\ {\bf Case 2.}  $p$ is odd. 

By  Proposition \ref{NEC} and Corollary \ref{comp}, $N_{\rm E}({\rm GR}(p^2,s), mp)\geq  N_{\rm E}({\rm GR}(p^2,s), p)=2 > 1$.
	
Conversely, we assume that $m=1$ and $p=2$. Then, by  Proposition \ref{NEC},	$ N_{\rm E}({\rm GR}(p^2,s),mp )=N_{\rm E}({\rm GR}(4,s),2 )= 1$.
\end{proof}
 
Some examples of the numbers of Euclidean self-dual cyclic codes of small lengths over $\mathbb{Z}_4$, over $\mathbb{Z}_9$, and over ${\rm GR}(4,2 )$,  are given in Tables \ref{T1},   \ref{T11}, and \ref{T2}, respectively. 

\begin{table}[!hbt]
\centering 
\caption{Number  of Euclidean Self-Dual Cyclic Codes over $\mathbb{Z}_4$}
\vskip.5em
\begin{tabular}{|c|c|c|c|c|c|c|c|c|c|c|c|c|c|c|c|}
\hline
$n$ &$1$ &$2$ &$3$&$4$&$5$&$6$&$7$&$8$&$9$&$10$\\ \hline
$N_{\rm E}(\mathbb{Z}_4, n)$&
1& 1& 1& 3& 1& 3& 1& 11& 1& 5\\\hline\hline

$n$ &$11$&$12$&$13$&$14$&$15$&$16$ &$17$ &$18$&$19$&$20$\\ \hline
$N_{\rm E}(\mathbb{Z}_4, n)$ & 1& 21& 1& 13& 1&
 59& 1& 27& 1& 63\\ \hline\hline

$n$&$21$&$22$&$23$&$24$&$25$&$26$&$27$&$28$&$29$&$30$\\ \hline
$N_{\rm E}(\mathbb{Z}_4, n)$& 1& 33& 1& 341& 1& 65& 1& 339& 1& 315\\ \hline\hline

$n$  &$31$ &$32$ &$33$&$34$&$35$&$36$&$37$&$38$&$39$&$40$\\ \hline
$N_{\rm E}(\mathbb{Z}_4, n)$& 1& 1019& 1& 289& 1& 1533& 1& 513& 1& 3751\\
\hline
\end{tabular}
\label{T1}
\end{table}

\begin{table}[!hbt]
\centering 
\caption{Number  of Euclidean Self-Dual Cyclic Codes over $\mathbb{Z}_9$}
\vskip.5em
\begin{tabular}{|c|c|c|c|c|c|c|c|c|c|c|c|c|c|c|c|}
\hline
$n$ &$1$ &$2$ &$3$&$4$&$5$&$6$&$7$&$8$&$9$&$10$\\ \hline
$N_{\rm E}(\mathbb{Z}_9, n)$&1& 1& 2& 1& 1& 4& 1& 1& 8& 1
 \\\hline\hline

$n$ &$11$&$12$&$13$&$14$&$15$&$16$ &$17$ &$18$&$19$&$20$\\ \hline
$N_{\rm E}(\mathbb{Z}_9, n)$ & 1& 16& 1& 1& 20& 1& 1& 64& 1& 1 \\ \hline\hline

$n$&$21$&$22$&$23$&$24$&$25$&$26$&$27$&$28$&$29$&$30$\\ \hline
$N_{\rm E}(\mathbb{Z}_9, n)$& 56& 1& 1& 544& 1& 1& 242& 1& 1& 400\\ \hline\hline

$n$  &$31$ &$32$ &$33$&$34$&$35$&$36$&$37$&$38$&$39$&$40$\\ \hline
$N_{\rm E}(\mathbb{Z}_9, n)$&  1& 1& 1472& 1& 1& 2560& 1& 1& 15488& 1 \\
\hline
\end{tabular}
\label{T11}
\end{table}

\begin{table}[!hbt]
\centering
\caption{Number  of Euclidean Self-Dual Cyclic Codes over ${\rm GR}(4,2 )$}
\vskip.5em
\begin{tabular}{|c|c|c|c|c|c|c|c|c|c|c|c|c|c|c|c|}
\hline
$n$ &$1$ &$2$ &$3$&$4$&$5$&$6$&$7$&$8$&$9$&$10$\\ \hline
$N_{\rm E}({\rm GR}(4,2),n )$&1& 1& 1& 5& 1& 9& 1& 37& 1& 25
 \\\hline\hline

$n$ &$11$&$12$&$13$&$14$&$15$&$16$ &$17$ &$18$&$19$&$20$\\ \hline
$N_{\rm E}({\rm GR}(4,2),n )$ &  1& 225& 1& 69& 1& 677& 1& 621& 1& 2205 \\ \hline\hline

$n$&$21$&$22$&$23$&$24$&$25$&$26$&$27$&$28$&$29$&$30$\\ \hline
$N_{\rm E}({\rm GR}(4,2),n )$&1&1029& 1& 29193& 1& 4225& 1& 22125& 1& 99225 \\ \hline\hline

$n$  &$31$ &$32$ &$33$&$34$&$35$&$36$&$37$&$38$&$39$&$40$\\ \hline
$N_{\rm E}({\rm GR}(4,2),n )$&  1& 174757& 1& 83521& 1& 995625& 1&
262149& 1& 4302397\\
\hline
\end{tabular}
 
\label{T2}
\end{table}
In general, the numbers of Euclidean self-dual cyclic codes of any  length  over ${\rm GR}(p^2,s)$, where $p$ is a prime and $s$ is a positive integer, can be computed using the formula in   Corollary~\ref{N-cyclic2}, Proposition \ref{NEC},  Theorem \ref{NHC}, and  Corollary \ref{comp}.

\subsection{A Note on Euclidean Self-dual Cyclic Codes of Even Length over $\mathbb{Z}_4$}
In this subsection, we reconsider   cyclic codes of even length over $\mathbb{Z}_4$ which have been studied in \cite{B2003} and  \cite{DL2006}. This case can be viewed as a special case of the previous two subsections  where  $p=2$ and $s=1$.  We discovered that   \cite[Lemma 7, Lemma 9, and Corollary 2]{B2003} and  \cite[Lemma 5.2, Corollary 5.4, Proposition 5.8, Corollary 5.9, and Section 6]{DL2006}     contain   errors.   Corrections to these errors are discussed as follows.

 Let $n$ be an even  positive integer written as $n=m2^a$, where $m$ is odd and $a\geq 1$. Then, in the decomposition (\ref{DEC}), we have $J_0=\{0\}$ and every  cyclic code  $C$ of length $n$ over  $\mathbb{Z}_4$  can be viewed as 
\begin{align}\label{C4}
C\cong  C_0 \times \prod_{h \in   J_1 } C_h \times \prod_{k\in  J_2} C_k,
\end{align}
where $C_0$, $C_k$, and $C_k$ are cyclic codes of length $2^a$ over   appropriate Galois extensions of  $\mathbb{Z}_4$.

We note that all the errors in \cite{B2003} and \cite{DL2006} are caused by misinterpretation of the orthogonality in   \cite[Lemma 6]{B2003} and \cite[Equation (21)]{DL2006} (see also (\ref{eq-ortho})). This effects  the components in the  decomposition of the Euclidean dual of $C$ in  (\ref{C4}) that relates to the $2$-cyclotomic cosets of elements in $J_1$.  The errors are pointed out in terms of  our notations. The readers  may refer to the original statements   via the number cited. The corrections to these  are discussed as well.

\begin{itemize}
\item  The Euclidean dual of $C$  in (\ref{C4}).

In \cite[Lemma 7 and  Lemma 9]{B2003} and  \cite[Lemma 5.2 and Corollary 5.4]{DL2006}, an incorrect  interpretation has been made as follows.

\hskip1em \parbox{0.8\textwidth}{\it 
\begin{align*} \text{\lq\lq}
C^{\perp_{\rm E}}\cong   C_0^{\perp_{\rm E}} \times \prod_{h \in    J_1 } C_h^{\perp_{\rm E}} \times \prod_{k\in  J_2} C^{\perp_{\rm E}}_{m-k\,({\rm mod}\,m)}. \text{\rq\rq}
\end{align*}
}

The Euclidean dual of $C$  in (\ref{C4})  can be viewed as a special case of   Proposition \ref{char}, where   $p=2$ and $s=1$, {\it i.e.},  
 \begin{align*} 
C^{\perp_{\rm E}}\cong  C_0^{\perp_{\rm E}} \times \prod_{h \in    J_1 } C_h^{\perp_{\rm H}} \times \prod_{k\in  J_2} C^{\perp_{\rm E}}_{m-k\,({\rm mod}\,m)}.
\end{align*}

 \item The     number of Euclidean self-dual cyclic codes   over $\mathbb{Z}_4$.

Since the determination of the  Euclidean dual of a cyclic code in  \cite{B2003} and \cite{DL2006}  is not correct, in \cite[Lemma 9]{B2003}  and  \cite[Proposition 5.8]{DL2006}, an incorrect statement about  the number of Euclidean self-dual cyclic codes of length $n=m2^a$ over $\mathbb{Z}_4$ has been proposed as follows.

\vskip1em
\hskip2em \parbox{0.85\textwidth}{\it ``  
	The     number of Euclidean self-dual cyclic codes of length $n=m2^a$ over $\mathbb{Z}_4$ is 
	\[    \prod_{h\in \{0\}\cup J_1} N_{\rm E}( {\rm GR}(4,m_h),2^a ) \times \prod_{k\in J_2^\prime} N({\rm GR}(4,m_k),2^a  ).\text{\rq\rq}\] 
}
\vskip1em
The correct statement can be viewed as a special case of Corollary \ref{comp} where $p=2$ and $s=1$.
 \begin{corollary}
	The   number of Euclidean self-dual codes of length $n=m2^a$ over $\mathbb{Z}_4$ is 
	\[   N_{\rm E}(\mathbb{Z}_4,2^a )  \times \prod_{h\in J_1} N_{\rm H}( {\rm GR}(4, m_h),2^a ) \times \prod_{k\in J_2^\prime} N({\rm GR}(4, m_k),2^a  ) .\]
	where  
  the empty product is regarded as $1$.
\end{corollary}

\item A unique Euclidean self-dual code of length $n=2m$ over $\mathbb{Z}_4$.

In  \cite[Corollary 2]{B2003} and  \cite[Corollary 5.9]{DL2006},   it has been claimed that
  
\vskip1em
\hskip2em \parbox{0.85\textwidth}{\it ``
If there exists an integer $e$ such that $-1\equiv 2^e\,{\rm mod}\, n$, then there is only one Euclidean self-dual cyclic code of length $2n$ where $n$ is odd, namely $2(\mathbb{Z}_4)^{2n}$.''
}
\vskip1em
 This claim is not correct and 
  the correct statement can be viewed as a special case of  Corollary \ref{pm} with $p=2$ and $s=1$.  It can be restated as follows.

\begin{corollary}  Let $m$ be an odd positive  integer. 
	Then 
	$ N_{\rm E}(\mathbb{Z}_4,2m )= 1$ if and only if $m=1$.
\end{corollary}

 \end{itemize}

This error in \cite{DL2006} led to an incorrect claim   that there is only one Euclidean self-dual cyclic codes of length $n$ in the cases where $n=6$ and   $n=10$  (see \cite[Section 6]{DL2006}).  Actually, in Table \ref{T1}, the numbers of such codes are $3$ and $5$, respectively. Here,    the    codes are also provided.

\begin{example}  The  Euclidean self-dual cyclic codes of length $6$ and  $10$ over $\mathbb{Z}_4$ are determined as follows. 
	\begin{itemize}
		\item $n=6$. Then $m=3$ and the $2$-cyclotomic cosets modulo $3$ are $\{0\}$ and $\{1,2\}$. Both the $2$-cyclotomic cosets are self-inverse, which implies that  $I_0=\{0\}$ and $I_1=\{1,2\}$, and hence, $J_0=\{0\}$ and $J_1=\{1\}$. Therefore, 
		\[\mathbb{Z}_4[X]/\langle X^6-1\rangle\cong R_4(u,1)\times R_4(u,2).\]
By Proposition \ref{NEC} and Theorem \ref{NHC}, $N_{\rm E}(\mathbb{Z}_4,2)=1$ and  $N_{\rm H}({\rm GR}(4,2), 2)=3$. 
Therefore, by Corollary \ref{comp}, we have $N_{\rm E}(\mathbb{Z}_4,6)=3$, given by the codes $\langle 2\rangle \times \langle 2\rangle$, $\langle 2\rangle \times \langle 1+u+2\xi\rangle$, and $\langle 2\rangle \times \langle1+u+2\xi^2\rangle$, where $\xi$  is the generator of a Teichm\"{u}ller set $\mathcal{T}_2$. 
		
		\item $n=10$.  Then $m=5$ and the $2$-cyclotomic cosets modulo $5$ are $\{0\}$ and $\{1,2,3,4\}$. Both the $2$-cyclotomic cosets are self-inverse, which implies that  $I_0=\{0\}$ and $I_1=\{1,2,3,4\}$, and hence, $J_0=\{0\}$ and $J_1=\{1\}$. Therefore, 
		\[\mathbb{Z}_4[X]/\langle X^{10}-1\rangle\cong R_4(u,1)\times R_4(u,4).\]
By Proposition \ref{NEC} and Theorem \ref{NHC}, $N_{\rm E}(\mathbb{Z}_4,2)=1$ and  $N_{\rm H}({\rm GR}(4,4), 2)=5$. 
Therefore, by Corollary \ref{comp}, we have $N_{\rm E}(\mathbb{Z}_4,10)=5$.
	\end{itemize}
\end{example}

\section{Conclusion}\label{sec5}
The complete characterization and enumeration   of Hermitian self-dual cyclic codes of length $p^a$ over ${\rm GR}({p^2},s)$ has been established.  Using the Discrete Fourier Transform decomposition, we have characterized the structure of Euclidean self-dual cyclic codes of any length   over ${\rm GR}({p^2},s)$.  The enumeration of such codes has been given through this decomposition, our results on Hermitian self-dual codes, and some  known results on cyclic codes of   length  $p^a$ over ${\rm GR}({p^2},s)$.   Based on the established  characterization and enumeration, we have   corrected    mistakes in some earlier results on  Euclidean self-dual cyclic codes of even length over $\mathbb{Z}_4$ in \cite{B2003}  and \cite{DL2006}.  

\section*{Acknowledgements}
Parts of this work were done when S. Jitman was a research fellow and E. Sangwisut was a visiting student  at  the Division of Mathematical Sciences,
  School of Physical and Mathematical Sciences, Nanyang Technological
  University, Singapore. They would like to thank Nanyang Technological
  University for the support and hospitality.

\end{document}